\documentclass[12pt]{amsart}
\usepackage{amssymb}
\usepackage{amsbsy}
\usepackage{amscd}

\theoremstyle{plain}
\newtheorem{thm}{Theorem}[section]
\newtheorem{prop}[thm]{Proposition}

\theoremstyle{definition}
\newtheorem{dfn}[thm]{Definition}
\newtheorem{ex}[thm]{Example}

\numberwithin{equation}{section}

\newcommand{\G}{\Gamma}
\newcommand{\g}{\gamma}

\newcommand{\sm}{\left(\begin{smallmatrix}}
\newcommand{\esm}{\end{smallmatrix}\right)}
\newcommand{\mm}{\begin{pmatrix}}
\newcommand{\emm}{\end{pmatrix}}

\newcommand{\mbb}{\mathbb}

\newcommand{\gl}{\lambda}

\newcommand{\gL}{\Lambda}

\newcommand{\bi}{\binom}

\newcommand{\mF}{\mathcal F}
\newcommand{\mH}{\mathcal H}

\newcommand{\fJ}{\mathfrak J}
\newcommand{\fK}{\mathfrak K}

\newcommand{\mQ}{\mathcal Q}

\newcommand{\bC}{\mathbb C}
\newcommand{\bR}{\mathbb R}
\newcommand{\bZ}{\mathbb Z}

\newcommand{\md}{\operatorname{\text{$\|$}}}
\newcommand{\fS}{\mathfrak S}

\begin{document}

\title{Dirichlet series of Rankin-Cohen brackets}


\author[YoungJu Choie]{YoungJu Choie$^\ast$}

\address{Department of Mathematics and
PMI (Pohang Mathematical Institute), POSTECH, Pohang 790-784,
Korea}

\email{yjc@postech.ac.kr}

\author[Min Ho Lee]{Min Ho Lee$^\dagger$}

\address{Department of Mathematics, University of Northern Iowa,
Cedar Falls, IA 50614, U.S.A.}

\email{lee@math.uni.edu}

\thanks{$^\ast$Supported in part by NRF20090083909 and NRF2009-0094069}
\thanks{$^\dagger$Supported in part by a summer fellowship from the University of
  Northern Iowa}


\begin{abstract}
Given modular forms $f$ and $g$ of weights $k$ and $\ell$, respectively,
their Rankin-Cohen bracket $[f,g]^{(k, \ell)}_n$ corresponding to a
nonnegative integer $n$ is a modular form of weight $k +\ell +2n$, and it
is given as a linear combination of the products of the form $f^{(r)}
g^{(n-r)}$ for $0 \leq r \leq n$.  We use a correspondence between
quasimodular forms and sequences of modular forms to express the Dirichlet
series of a product of derivatives of modular forms as a linear combination
of the Dirichlet series of Rankin-Cohen brackets.
\end{abstract}



\maketitle \pagestyle{plain}

\section{\bf{Introduction and statement of main results}}

Given modular forms $f$ and $g$ of weights $k$ and $\ell$, respectively,
their Rankin-Cohen bracket $[f,g]^{(k, \ell)}_n$ corresponding to a
nonnegative integer $n$ is a modular form of weight $k +\ell +2n$, and it
is given as a linear combination of the products of the form $f^{(r)}
g^{(n-r)}$ for $0 \leq r \leq n$ (see e.g.\ \cite{CM97}).  Although such
products are not modular forms, they are quasimodular forms.

Quasimodular forms generalize classical modular forms  and  first
introduced by Kaneko and Zagier in \cite{KZ95}. It appears
naturally in various places (see \cite{EO1, EO2, LR}, for
instance). One of the useful features of quasimodular forms is
that their derivatives are also quasimodular forms (see
\cite{BGHZ06, MR05}). In particular, derivatives of modular forms
are quasimodular forms. Since products of quasimodular forms are
quasimodular forms, it follows that the the products $f^{(r)}
g^{(n-r)}$ considered above are quasimodular forms.  As in the
case of modular forms, we can consider the Dirichlet series
associated to quasimodular forms by using their Fourier
coefficients.  From the formula for Rankin-Cohen brackets it
follows that the Dirichlet series of the modular forms $[f,g]^{(k,
  \ell)}_n$ can be written as the Dirichlet series of the quasimodular
forms $f^{(r)} g^{(n-r)}$.  The goal of this paper is to
express the Dirichlet series of a product of derivatives of modular forms
in terms of the Dirichlet series of Rankin-Cohen brackets.  More precisely,
we prove the following theorem:

\begin{thm}\label{maintheorem}
Given a discrete subgroup $\G$, containing translations,  of
$SL(2, \bR)$, let $\phi$ and $\psi$ be modular forms for $\G$ with
width $h$ and  weights $\mu$ and $\nu$, respectively. Then the
Dirichlet series of the quasimodular form $\phi^{(m)} \psi^{(n)}$
can be written in the form
\begin{equation} \label{E:ju}
L  (\phi^{(m)} \psi^{(n)}, s) = \sum^{m +n}_{\ell =0}
a^{m,n}_{\mu,\nu} (\ell) L ([\phi, \psi]^{(\mu, \nu)}_{m+n -\ell},
s -\ell) ,
\end{equation}
where
\begin{align} \label{E:8q}
a^{m,n}_{\mu,\nu} (\ell) &= \frac {(2\pi i)^\ell n! (\mu +m-1)!
(\nu +n -1)! (\mu +\nu +2\ell -1)!} {(\mu +\ell -1)! (\mu +\nu +2m
+2n -\ell -1)! h^\ell}\\
&\hspace{.4in} \times  \sum^{\ell}_{j=0} (-1)^j \frac {(m +n
-\ell+j)! (2\ell +\mu +\nu -j-2)!} {j! (\ell -j)! (n -\ell +j)! (\nu +\ell
-j -1)!} \notag
\end{align}
for $0 \leq \ell \leq m+n$.
\end{thm}

The proof of this theorem is carried out by using a correspondence between
quasimodular forms and sequences of modular forms discussed in \cite{L09b}.
For example, each quasimodular form can be written as a linear combination
derivatives of a finite number of modular forms.

We would like to thank the referee for various helpful comments and
suggestions.

\section{\bf{Quasimodular and modular polynomials}} \label{S:qmp}

In this section we describe $SL(2, \bR)$-equivariant automorphisms of
spaces of polynomials introduced in \cite{L09b} which determine
correspondences between quasimodular polynomials and modular polynomials.

Let $\mH$ be the Poincar\'e upper half plane, and let $\mF$ be the ring of
holomorphic functions $f: \mH \to \bC$ satisfying the following growth
condition
\begin{equation} \label{E:mn}
|f(z)| \ll \biggl( \frac {{\rm Im}\, z} {1 + |z|^2} \biggr)^{-\nu}
\end{equation}
for some $\nu >0$ (see e.g.\ \cite[Section 17.1]{MR05} for a more precise
description of this condition).  We fix a nonnegative integer $m$ and
denote by $\mF_m [X]$ the complex vector space of polynomials in $X$ over
$\mF$ of degree at most $m$.  Thus $\mF_m [X]$ consists of polynomials of
the form
\begin{equation} \label{E:ge}
\Phi (z, X) = \sum^m_{r=0} \phi_r (z) X^r
\end{equation}
with $\phi \in \mF$ for $0 \leq r \leq m$.  If $\Phi (z, X) \in \mF_m [X]$
is as in \eqref{E:ge} and if $\gl$ is an integer with $\gl > 2m$, we set
\begin{equation} \label{E:xw}
(\gL^m_\gl \Phi) (z, X) = \sum^m_{r=0} \phi^\gL_r (z) X^r, \quad
(\Xi^m_\gl \Phi) (z, X) = \sum^m_{r=0} \phi^\Xi_r (z) X^r
\end{equation}
where
\begin{align}
\phi^\gL_r &= \frac 1{r!} \sum^{m-r}_{\ell=0} \frac {1} {\ell! (\gl -2r
-\ell-1)!}
\phi^{(\ell)}_{m -r -\ell} \label{E:5rr}\\
\label{E:5ow}
\phi^\Xi_r &= (\gl + 2r -2m -1) \sum^{r}_{\ell=0} \frac {(-1)^\ell} {\ell!}
(m-r +\ell)!\\
&\hspace{2.0in} \times (2r +\gl -2m -\ell-2)! \phi^{(\ell)}_{m-r +\ell} ,
\notag
\end{align}
for each $r \in \{ 0, 1, \ldots, m \}$.  Then it can be shown that
the resulting maps
\[ \gL^m_\gl, \Xi^m_\gl: \mF_m [X] \to \mF_m [X] \]
are complex linear isomorphisms with
\[ (\Xi^m_\gl)^{-1} = \gL^m_\gl \]
(see \cite{L09b}).

The group $SL(2, \bR)$ acts on the Poincar\'e upper half plane $\mH$ as
usual by linear fractional transformations, so that
\[ \g z = \frac {az+b} {cz+d} \]
for all $z \in \mH$ and $\g = \sm a&b\\ c&d \esm \in SL(2,\mbb R)$.  For
the same $z$ and $\g$, by setting
\begin{equation} \label{E:kz}%
\fJ (\g, z) = cz+d, \quad \fK (\g,z) = \frac c{cz+d} ,
\end{equation}
we obtain the maps $\fJ, \fK: SL(2, \bR) \times \mH \to \bC$
which satisfy
\[ \fJ (\g \g', z) = \fJ (\g, \g' z) \fJ (\g', z), \quad \fK (\g, \g' z) =
\fJ (\g', z)^2 (\fK (\g \g',z) - \fK (\g',z)) \]
for all $z \in \mH$ and $\g, \g' \in SL(2,\mbb R)$.

Given $\Phi (z, X) \in \mF_m [X]$ as in \eqref{E:ge} and elements $f \in
\mF$, $\g \in SL(2, \bR)$ and $\gl \in \bZ$, we set
\begin{equation} \label{E:ts0}
(f \mid_\gl \g) (z) = \fJ (\g, z)^{-\gl} f (\g z) ,
\end{equation}
\begin{equation} \label{E:us}
(\Phi \mid^X_\gl \g) (z, X) = \sum^m_{r=0} (\phi_r \mid_{\gl +2r}
\g) (z) X^r ,
\end{equation}
\begin{equation} \label{E:js}
(\Phi \md_\gl \g) (z, X) = \fJ (\g, z)^{-\gl} \Phi (\g z, \fJ (\g,
z)^2 ( X - \fK (\g, z)))
\end{equation}
for all $z \in \mH$.  Then $\mid_\gl$ determines an action of
$SL(2, \bR)$ on $\mH$ as usual, and the other two operations
$\mid^X_\gl$ and $\md_\gl$ determine actions of the same group on
$\mF_m [X]$.  If $\gL^m_\gl$ and $\Xi^m_\gl$ are the linear automorphisms
of $\mF_m [X]$ given by \eqref{E:xw}, then it is known that
\begin{equation} \label{E:9u}
((\gL^m_\gl \Phi) \md_\gl \g) (z,X) = \gL^m_\gl (\Phi \mid^X_{\gl
-2m} \g) (z,X) ,
\end{equation}
\begin{equation} \label{E:9w}
((\Xi^m_\gl \Phi) \mid^X_{\gl -2m} \g) (z,X) = \Xi^m_\gl (\Phi
\md_\gl \g) (z,X)
\end{equation}
for all $\g \in SL(2, \bR)$ (cf.\ \cite{L09b}).

We now consider a discrete subgroup $\G$ of $SL(2, \bR)$.  Then an element
$f \in \mF$ is a {\em modular form for $\G$ of weight $\gl$\/} if it
satisfies
\[
f \mid_\gl \g = f
\]
for all $\g \in \G$, where the operation $\mid_\gl$ is given by
\eqref{E:ts0}.

\begin{dfn} \label{D:ns}
Let $\mid^X_\gl$ and $\md_\gl$ with $\gl \in \bZ$ be the
operations in \eqref{E:us} and \eqref{E:js}.

(i) An element $F (z,X) \in \mF_m [X]$ is a {\em modular
polynomial for
  $\G$ of weight $\gl$ and degree at most $m$\/} if it satisfies
\[
F \mid^X_\gl \g = F
\]
for all $\g \in \G$.

(ii) An element $\Phi (z,X) \in \mF_m [X]$ is a {\em quasimodular
polynomial for $\G$ of weight $\gl$ and degree at most $m$\/} if
it satisfies
\[
\Phi \md_\gl \g = \Phi
\]
for all $\g \in \G$.
\end{dfn}

We denote by $MP^m_\gl (\G)$ and $QP^m_\gl (\G)$ the spaces of
modular and quasimodular, respectively, polynomials for $\G$ of
weight $\gl$ and degree at most $m$. If a polynomial $F (z,X) \in
\mF_m [X]$ of the form
\[ F (z,X)  = \sum^m_{r=0} f_r (z) X^r \]
belongs to  $MP^m_\gl (\G)$, from \eqref{E:us} and Definition
\ref{D:ns}(i) we see that
\begin{equation} \label{E:k7}
f_r \in M_{\gl +2r} (\G)
\end{equation}
for $0 \leq r \leq m$.  From \eqref{E:9u}, \eqref{E:9w} and Definition
\ref{D:ns} it follows that the automorphisms $\gL^m_\gl$ and $\Xi^m_\gl$ of
$\mF_m [X]$ induce the isomorphisms
\begin{equation} \label{E:nw}
\gL^m_\gl: MP^m_{\gl-2m} (\G) \to QP^m_\gl (\G), \quad \Xi^m_\gl:
QP^m_\gl (\G) \to MP^m_{\gl-2m} (\G)
\end{equation}
for each $\gl \geq 2m$.

\begin{ex}
Given an integer $\gl \geq 2m$, we consider two modular forms
\[ \xi \in M_{\gl -2m} (\G), \quad \eta \in M_{\gl -2m +2} (\G)
\]
and the associated modular polynomial
\[ F (z,X) = \sum^m_{r=0} f_r (z) X^r \in MP_{\gl -2m} (\G)\]
with
\[ f_r = \begin{cases}
\xi & \text{if $r=0$;}\\
\eta & \text{if $r =1$;}\\
0 & \text{if $2 \leq r \leq m$.}
\end{cases}\]%
If $\gL^m_\gl F (z,X) = \sum^m_{r=0} f^\gL_k (z) X^r \in QP^m_\gl
(\G)$ is the corresponding quasimodular polynomial, from \eqref{E:5rr} we
obtain
\begin{align*}
f^\gL_k &= \frac {f^{(m-k)}_0} {k! (m-k)! (\gl -k -m -1)!} +
\frac {f^{(m-k -1)}_1} {k! (m-k -1)! (\gl -k -m)!}\\
&= \frac {(\gl -k -m) \xi^{(m-k)} + (m-k) \eta^{(m -k -1)}} {k!
(m-k)! (\gl -k -m)!}.
\end{align*}
Thus we have
\begin{equation} \label{E:jy}
\gL^m_\gl F (z,X) = \sum^m_{k=0} \frac {(\gl -k -m)
\xi^{(m-k)} (z) + (m-k) \eta^{(m -k -1)} (z)} {k! (m-k)! (\gl -k -m)!} X^k .
\end{equation}
\end{ex}

\section{\bf{Quasimodular forms}}

In this section we discuss the correspondence between quasimodular
polynomials and quasimodular forms.  We also express the Dirichlet
series of a quasimodular form in terms of Dirichlet series of the modular
forms associated to the corresponding quasimodular polynomial.

Let $\mF$ be the ring of holomorphic functions on $\mH$ satisfying
\eqref{E:mn} as in Section \ref{S:qmp}, and let $\G$ be a discrete
subgroup of $SL(2, \bR)$ containing   translations.

\begin{dfn} \label{D:mq}
Given integers $m$ and $\gl$ with $m \geq 0$, an element $f \in \mF$ is a
{\em quasimodular form for $\G$ of weight $\gl$ and depth at most
$m$\/} if there are functions $f_0, \ldots, f_m \in \mF$ such that
\begin{equation} \label{E:qq1}
(f \mid_\gl \g) (z) = \sum^m_{r=0} f_r (z) \fK (\g, z)^r
\end{equation}
for all $z \in \mH$ and $\g \in \G$, where $\fK (\g, z)$ is as in
\eqref{E:kz} and $\mid_\gl$ is the operation in \eqref{E:ts0}.  We
denote by $QM^m_\gl (\G)$ the space of quasimodular forms for $\G$
of weight $\gl$ and depth at most $m$.
\end{dfn}

Quasimodular forms correspond to quasimodular polynomials as described
below.  If $0 \leq \ell \leq m$, we consider the complex linear map
\[ \fS_\ell: \mF_m [X] \to \mF \]
defined by
\[ \fS_\ell \biggl( \sum^m_{r=0} \phi_r (z) X^r \biggr) = \phi_\ell (z) \]
for all $z \in \mH$.  Then it can be shown (cf.\ \cite{CL08}) that
\[ \fS_\ell (QP^m_\gl (\G)) \subset QM^{m-\ell}_{\gl-2\ell} (\G) ;\]
hence we obtain the map
\begin{equation} \label{E:8t}
\fS_\ell: QP^m_\gl (\G) \to QM^{m-\ell}_{\gl-2\ell} (\G)
\end{equation}
for each $\ell$.  For $\ell =0$ it is known that the map
\[ \fS_0: QP^m_\gl (\G) \to QM^m_\gl (\G) \]
is an isomorphism whose inverse is the map
\begin{equation} \label{E:sp}
\mQ_\gl^m: QM^m_\gl (\G) \to QP^m_{\gl} (\G) .
\end{equation}
defined by
\[
(\mQ_\gl^m f) (z,X) = \sum^m_{r=0} f_r (z) X^r
\]
for a quasimodular form $f \in QM^m_\gl (\G)$ and functions $f_0, \ldots,
f_m \in \mF$ as in \eqref{E:qq1}.

Let $\psi \in QM^m_\gl (\G)$ be a quasimodular form whose Fourier
expansion is of the form
\begin{equation} \label{E:4h}
\psi (z) = \sum^\infty_{k =0} a_k e^{2\pi i kz/h}
\end{equation}
with $h \in \bR$, so that the corresponding Dirichlet series is given by
\begin{equation} \label{E:6p}
L (\psi, s) = \sum^\infty_{n=1} \frac {a_n} {n^s}
\end{equation}
where it converges when ${\rm Re} \, s \gg 0.$
For $0 \leq r \leq
m$, using \eqref{E:k7} and the isomorphisms in \eqref{E:nw} and
\eqref{E:sp}, we see that the function $(\fS_r \circ \Xi^m_\gl
\circ \mQ^m_\gl) \psi$ is a modular form belonging to $M_{\gl -2m
+2r} (\G)$. We now set
\begin{equation} \label{E:k8}
f^\psi_{r} = (\fS_r \circ \Xi^m_\gl \circ \mQ^m_\gl) \psi \in
M_{\gl -2m +2r} (\G) ,
\end{equation}
and assume that its Fourier expansion is given by
\begin{equation} \label{E:5h}
f^\psi_{r} (z) = \sum^\infty_{k =0} c_{r,k} e^{2\pi i kz/h} .
\end{equation}
Thus the corresponding Dirichlet series can be written as
\begin{equation} \label{E:7p}
L (f^\psi_{r}, s) = \sum^\infty_{n=1} \frac {c_{r,n}} {n^s},
\end{equation}
where it converges for ${\rm Re} \, s \gg 0.$

\begin{prop} \label{P:y4}
The Dirichlet series in \eqref{E:6p} and \eqref{E:7p} satisfy the
relation
\begin{equation} \label{E:jf}
L (\psi, s) = \sum^m_{\ell =0} \frac {(2\pi i)^\ell} {\ell! (\gl
-\ell -1)! h^\ell} L (f^\psi_{m -\ell}, s -\ell),
\end{equation}
where it converges for  ${\rm Re} \, s \gg 0.$

\end{prop}

\begin{proof}
From \eqref{E:k8} we see that
\[ ((\Xi^m_\gl \circ \mQ^m_\gl) \psi) (z, X) = \sum^m_{r=0}
f^\psi_{r} (z) X^r \in MP^m_{\gl -2m} (\G) .\]%
Applying the isomorphism $\gL^m_\gl$ in \eqref{E:nw} to this
relation and using \eqref{E:xw} and \eqref{E:5rr}, we have
\[ (\mQ^m_\gl \psi) (z, X) = ((\gL^m_\gl \circ \Xi^m_\gl \circ
\mQ^m_\gl) \psi) (z, X) = \sum^m_{r=0} \psi_r (z) X^r \in
QP^m_\gl (\G) ,\]%
where
\[ \psi_r = \frac {1}{r!}\sum^{m -r}_{\ell =0} \frac 1{\ell!
(\gl -2r -\ell -1)!} (f^\psi_{m -r -\ell})^{(\ell)} \]%
for $0 \leq r \leq m$.  From this and the fact that $(\mQ^m_\gl)^{-1} =
\fS_0$ we obtain
\[ \psi = \fS_0 (\mQ^m_\gl \psi) = \psi_0 = \sum^m_{\ell =0} \frac
1{\ell! ( \gl -\ell -1)!} (f^\psi_{m -\ell})^{(\ell)} .\]%
Using \eqref{E:5h}, we have
\[ (f^\psi_{m -\ell})^{(\ell)} (z) = \sum^\infty_{k =0}
c_{m -\ell,k} \biggl( \frac {2\pi i k} {h} \biggr)^\ell e^{2\pi i
kz/h} ;\]%
hence we obtain
\[ \psi (z) = \sum^\infty_{k =0} \sum^m_{\ell =0}
\frac {(2\pi i k)^\ell c_{m -\ell,k}} {\ell! ( \gl -\ell -1)!
h^\ell} e^{2\pi i kz/h}  .\]%
Comparing this with \eqref{E:4h}, we have
\[ a_k = \sum^m_{\ell =0}
\frac {(2\pi i k)^\ell c_{m -\ell,k}} {\ell! ( \gl -\ell -1)!
h^\ell} \]%
for $k \geq 0$.  Thus from \eqref{E:6p} and \eqref{E:7p} we see
that
\begin{align*}
L (\psi, s) &= \sum^\infty_{n=1} \sum^m_{\ell =0} \frac {(2\pi i
n)^\ell c_{m -\ell,n}} {\ell! ( \gl -\ell -1)! h^\ell n^s}\\
&= \sum^m_{\ell =0} \sum^\infty_{n=1} \frac {(2\pi i)^\ell
c_{m -\ell,n}} {\ell! ( \gl -\ell -1)! h^\ell n^{s -\ell}}\\
&= \sum^m_{\ell =0} \frac {(2\pi i)^\ell} {\ell! ( \gl -\ell -1)!
h^\ell} L (f^\psi_{m-\ell}, s -\ell) ;
\end{align*}
hence the proposition follows.
\end{proof}

\section{\bf{Proof of Theorem 1.1}}

We first recall that the Rankin-Cohen brackets are the bilinear maps
\[ [\; , \;]^{(k, \ell)}_w: M_{k} (\G) \times M_{\ell} (\G) \to
M_{k +\ell +2w} (\G) \]%
defined by
\begin{equation} \label{E:bp}
[f , g]^{(k, \ell)}_w = \sum^w_{r=0} (-1)^r \bi {k +w -1} {w-r}
\bi {\ell +w -1} {r} f^{(r)} g^{(w-r)}
\end{equation}
for $k, \ell \in \bZ$, $w \geq 0$, $f \in M_{k} (\G)$ and $g \in M_{\ell}
(\G)$ (see e.g.\ \cite{Z94}, \cite{CM97}).  It is known that the
Rankin-Cohen brackets are unique up to constant.  More precisely, if
\[ B_w: M_{k} (\G) \times M_{\ell} (\G) \to
M_{k +\ell +2w} (\G) \]%
is a bilinear differential operator, there is a constant $c \in
\bC$ such that
\[ B_w(f,g) = c [f , g]^{(k, \ell)}_w \]
for all $(f,g) \in M_{k} (\G) \times M_{\ell} (\G)$ .

\smallskip

\noindent
\textbf{Proof of Theorem \ref{maintheorem}.}
Given nonnegative integers $\mu$ and $\nu$, we consider modular forms $\phi
\in M_{\mu} (\G)$ and $\psi \in M_{\nu} (\G)$.  Then their derivatives
$\phi^{(m)}$ and $\psi^{(n)}$ with $m, n \geq 0$ are quasimodular forms
with
\[ \phi^{(m)} \in QM^m_{\mu +2m}, \quad \psi^{(n)} \in QM^n_{\nu +2n} ;\]
hence we see that $\phi^{(m)} \psi^{(n)}$ is a quasimodular form belonging
to $QM^{m +n}_{\mu +\nu +2m +2n} (\G)$.  By setting $\gl = \mu +2m$, $\xi =
\phi$ and $\eta =0$ in \eqref{E:jy}, we obtain
\begin{equation} \label{E:6g}
\gL^m_{\mu +2m} \Phi (z,X) = \sum^m_{k=0} \frac
{\phi^{(m-k)}} {k! (m-k)! (\mu +m -k -1)!} X^k \in QP^m_{\mu +2m}
(\G) ,
\end{equation}
\[ ((\fS_0 \circ \gL^m_{\mu +2m})
\Phi) (z) = \frac {\phi^{(m)}} {m! (\mu +m-1)!} \in
QM^m_{\mu +2m} (\G) .\]%
Similarly, if $\psi \in M_{\nu} (\G)$, we have
\[ \gL^n_{\nu +2n} \Psi (z,X) = \sum^n_{\ell=0}
\frac {\psi^{(n-\ell)} (z)} {\ell! (n-\ell)! (\nu+n -\ell -1)!}
X^\ell \in QP^n_{\nu +2n} (\G) ,\]%
\[ ((\fS_0 \circ \gL^n_{\nu +2n})
\Psi) (z) = \frac {\psi^{(n)}} {n! (\nu +n-1)!} \in
QM^n_{\nu +2n} (\G) .\]%
Thus, if we set
\[ F (z,X) = (\gL^m_{\mu +2m} \Phi (z,X)) \cdot (\gL^n_{\nu +2m} \Psi
(z,X)) ,\]
it is a quasimodular polynomial belonging to $QP^{m +n}_{\mu
+\nu +2m +2n} (\G)$, and we obtain
\begin{align*}
&F (z,X)\\
&=\sum^m_{k=0} \sum^n_{\ell=0} \frac {\phi^{(m-k)} (z)
\psi^{(n-\ell)} (z)} {k! \ell!(m-k)!  (n-\ell)! (\mu +m -k -1)!
(\nu +n -\ell -1)!} X^{k +\ell}\\
&= \sum^{m +n}_{r=0} \sum^r_{k=0} K^{m,n;\mu,\nu}_{k,r}
\phi^{(m-k)} (z) \psi^{(n-r +k)} (z) X^r \in QP^{m +n}_{\mu +\nu
+2m +2n} (\G) ,
\end{align*}
where
\[ K^{m,n;\mu, \nu}_{k,r} = \frac {1} {k! (r-k)! (m-k)!
(n-r +k)!(\mu +m -k -1)! (\nu +n -r +k -1)!} \]%
for $0 \leq k \leq r \leq m+n$; here we assume that $\phi^{(a)} =0$ and
$\psi^{(b)} =0$ for $a,b <0$.  Using \eqref{E:xw} and \eqref{E:5ow}, we have
\[ \Xi^{m+n}_{\mu +\nu +2m +2n} F (z,X) = \sum^{m+n}_{j =0}
\phi^\Xi_j(z) X^j \in MP^{m+n}_{\mu +\nu} (\G),\]%
where
\begin{align*}
\phi^\Xi_\ell &= (\mu +\nu + 2\ell -1) \sum^{\ell}_{j=0}
\frac {(-1)^j} {j!} (m +n-\ell +j)!\\
&\hspace{2.0in} \times (2\ell +\mu +\nu -j-2)! \phi^{(j)}_{m +n
-\ell +j}
\end{align*}
with
\[ \phi_r = \sum^r_{k=0} K^{m,n;\mu, \nu}_{k,r} \phi^{(m-k)}
\psi^{(n-r +k)} \]%
for $0 \leq \ell, r \leq m+n$.  However, we have
\begin{align*}
\phi^{(j)}_{m +n -\ell +j} &= \sum^{m +n -\ell +j}_{k=0}
K^{m,n;\mu, \nu}_{k,m +n -\ell
+j} (\phi^{(m-k)} \psi^{(\ell +k -m -j)})^{(j)}\\
&= \sum^{m +n -\ell +j}_{k=0} K^{m,n;\mu, \nu}_{k,m +n -\ell +j}
\sum^j_{p=0} \bi {j} {p} \phi^{(m-k +j -p)} \psi^{(\ell +k +p -m
-j)};
\end{align*}
Hence it follows that
\begin{align} \label{E:cfs}
\phi^\Xi_\ell &= (\mu +\nu  + 2\ell -1)\\
&\hspace{.3in} \times \sum^{\ell}_{j=0} \sum^{m +n
-\ell +j}_{k=0} \sum^j_{p=0}
\frac {(-1)^j} {j!} \bi {j} {p} (m +n-\ell +j)! \notag\\
&\hspace{.8in} \times (2\ell +\mu +\nu -j-2)!
K^{m,n;\mu, \nu}_{k,m +n -\ell +j} \notag\\
&\hspace{1.4in} \times \phi^{(m-k +j -p)} \psi^{(\ell +k +p -m
-j)} .\notag
\end{align}
for $0 \leq \ell \leq m+n$.  Since $\phi^\Xi_\ell \in M_{\mu +\nu
+2\ell} (\G)$, we obtain the bilinear differential operator
\[ (\phi, \psi) \mapsto \phi^\Xi_\ell : M_{\mu} (\G)
\times M_{\nu} (\G) \to M_{\mu +\nu +2\ell} (\G) \]%
on $M_{\mu} (\G) \times M_{\nu} (\G)$ for each $\ell$.  Thus,
using the uniqueness of Rankin-Cohen brackets, we see that
\begin{equation} \label{E:ba}
\phi^\Xi_\ell = b_\ell [\phi, \psi]^{(\mu, \nu)}_\ell
\end{equation}
for some $b_\ell \in \bC$.  Using $p=j$ and $k =m$, the
coefficient of $\phi \psi^{(\ell)}$ in \eqref{E:cfs} is given by
\begin{equation} \label{E:nx}
(\mu +\nu  + 2\ell -1) \sum^\ell_{j=0} \frac {(-1)^j} {j!} (m
+n-\ell +j)! (2\ell +\mu +\nu -j -2)! K^{m,n;\mu, \nu}_{m, m +n
-\ell +j} ,
\end{equation}
where
\[ K^{m,n;\mu, \nu}_{m, m +n -\ell +j} = \frac {1} {m! (n -\ell +j)!
(\ell-j)!(\mu -1)! (\nu +\ell -j -1)!} .\]%
On the other hand, by \eqref{E:bp} the coefficient of $\phi
\psi^{(\ell)}$ in \eqref{E:ba} is equal to
\[ \bi {\mu +\ell -1} {\ell} b_\ell = \frac {(\mu +\ell -1)! b_\ell}
{\ell! (\mu -1)!} .\]%
Comparing this with \eqref{E:nx}, we obtain
\begin{equation} \label{E:hd}
b_\ell = \frac {(\mu +\nu  + 2\ell -1) \ell!} {(\mu +\ell -1)! m!}
\sum^\ell_{j=0} (-1)^j \frac {(m +n-\ell +j)! (2\ell +\mu +\nu -j
-2)!} {j! (n -\ell +j)! (\ell-j)! (\nu +\ell -j -1)!}
\end{equation}
for $0 \leq \ell \leq m+n$.  Since we have
\[ \fS_0  F = \frac {\phi^{(m)} \psi^{(n)}}
{m! n! (\mu +m -1)! (\nu +n-1)!} \in QM^{m+n}_{\mu +\nu +2m +2n}
(\G) ,\]
it follows that
\[ (\fS_r \circ \Xi^{m+n}_{\mu +\nu +2m +2n} \circ \mQ^{m+n}_{\mu +\nu +2m
  +2n}) (\fS_0 F) = (\fS_r \circ \Xi^{m+n}_{\mu +\nu +2m +2n}) F =
\phi^\Xi_r \]
for $0 \leq r \leq m+n$.  Thus, using Proposition \ref{P:y4}, we obtain
\begin{align*}
L (\phi^{(m)} \psi^{(n)}, s) &= m! n! (\mu +m-1)! (\nu +n -1)!\\
&\hspace{.4in} \times \sum^{m +n}_{\ell =0} \frac {(2\pi i)^\ell}
{\ell! (\mu +\nu +2m +2n -\ell -1)! h^\ell} L (\phi^\Xi_{m+n
-\ell}, s -\ell)\\
&= m! n! (\mu +m-1)! (\nu +n -1)!\\
&\hspace{.4in} \times \sum^{m +n}_{\ell =0} \frac {(2\pi i)^\ell
b_\ell} {\ell! (\mu +\nu +2m +2n -\ell -1)! h^\ell} L ([\phi,
\psi]^{(\mu, \nu)}_{m+n -\ell}, s -\ell) ;
\end{align*}
hence the theorem follows from this and \eqref{E:hd}.

\smallskip


\section{\bf{Examples}}

In this section we consider two modular forms
\[ \phi \in M_{\mu} (\G), \quad \psi \in M_{\nu} (\G) ,\]
and provide examples of the formula \eqref{E:ju} for $(m,n) = (1,1), (2,0),
(0,2)$.

We first consider the case where $m = n =1$ by regarding the given modular
forms as modular polynomials
\[ \Phi (z,X) = \phi (z) \in MP^1_\mu (\G), \quad \Psi (z,X) = \psi (z) \in
MP^1_\nu (\G) .\]
Then  from \eqref{E:6g} we see that
\begin{align*}
(\gL^1_{\mu +2} &\Phi (z,X)) \cdot (\gL^1_{\nu +2} \Psi (z,X))\\
&= \frac {1} {\mu! \nu!} \biggl[ \phi' (z) \psi' (z) + \Bigl( \nu
\phi' (z) \psi (z) + \mu \phi (z) \psi' (z) \Bigr)X + \mu \nu \phi
(z) \psi (z) X^2 \biggr] ,
\end{align*}
which is a quasimodular form belonging to $QP^2_{\mu +\nu +4} (\G)$.  Thus,
if we set
\[ F (z,X) = \mu! \nu! (\gL^1_{\mu +2} \Phi (z,X)) \cdot (\gL^1_{\nu +2}
\Psi (z,X)) ,\]
we have
\[ F (z,X) = f_0 (z) + f_1 (z) X + f_2 (z) X^2 \in QP^2_{\mu +\nu +4} (\G)
,\]
where
\[ f_0 = \phi' \psi' = \fS_0 F \in QM^2_{\mu +\nu +4} (\G), \quad f_1 = \nu
\phi' \psi + \mu \phi \psi' , \quad f_2 = \mu \nu
\phi \psi .\]
Using \eqref{E:5ow}, we have
\[ (\Xi^2_{\mu +\nu +4} F) (z,X) = \sum^2_{r=0} f^\Xi_r (z) X^r ,\]
where
\begin{align*}
f^\Xi_0 &= 2 \mu \nu
(\mu +\nu -1)! \phi \psi,\\
f^\Xi_1 &= (\mu + \nu +1) (\mu +\nu -1)! (\mu -\nu) \Bigl( \mu \phi \psi'
-\nu \phi' \psi \Bigr) ,\\
f^\Xi_2 &= -(\mu +\nu +3) (\mu +\nu)! \\
&\hspace{.5in} \times \Bigl( \mu (\mu +1) \phi \psi'' -2 (\mu +1)
(\nu +1) \phi' \psi' +\nu (\nu +1) \phi'' \psi\Bigr).
\end{align*}
However, using \eqref{E:bp}, we have
\begin{align} \label{E:8r}
[\phi, \psi]^{(\mu, \nu)}_0 &= \phi \psi ,\\
[\phi, \psi]^{(\mu, \nu)}_1 &= \mu \phi \psi' - \nu \phi' \psi , \notag\\
[\phi, \psi]^{(\mu, \nu)}_2 &= \frac 12 \Bigl( \mu (\mu+1) \phi
\psi'' -2 (\mu +1) (\nu +1) \phi' \psi' +\nu (\nu +1) \phi'' \psi
\Bigr) \notag
\end{align}
Thus we see that
\begin{align*}
f^\Xi_0 &= 2 \mu \nu (\mu +\nu -1)! [\phi,
\psi]^{(\mu, \nu)}_0\\
f^\Xi_1 &= (\mu -\nu) (\mu +\nu +1) (\mu + \nu -1)! [\phi,
\psi]^{(\mu, \nu)}_1\\
f^\Xi_2 &= -2 (\mu +\nu +3) (\mu +\nu)! [\phi, \psi]^{(\mu,
\nu)}_2 .
\end{align*}
From this and \eqref{E:jf} with $\gl = \mu +\nu +4$ we obtain
\begin{align} \label{E:fp1}
L (\phi' \psi', s) &=  \sum^2_{j =0} \frac {(2\pi
  i)^j} {j! (\mu +\nu +3 -j)! h^j} L (f^\Xi_{2 -j}, s -j)\\
&= -\frac {2} {(\mu + \nu +2) (\mu + \nu +1)} L ([\phi,
\psi]^{(\mu, \nu)}_2, s) \notag\\
& \hspace{.5in}+ \frac {2\pi i (\mu - \nu)} {(\mu +\nu +2) (\mu
  +\nu) h} L([\phi, \psi]^{(\mu, \nu)}_1, s-1) \notag\\
& \hspace{.8in} + \frac {(2\pi i)^2 \mu \nu } {(\mu
  +\nu +1) (\mu +\nu) h^2} L ([\phi, \psi]^{(\mu, \nu)}_0, s -2). \notag
\end{align}

We now consider the case where $m=2$ and $n=0$ by regarding the given
modular forms as the modular polynomials
\[ \Phi (z,X) = \phi (z) \in MP^2_\mu (\G), \quad \Psi (z,X) = \psi (z) \in
MP^0_\nu (\G) , \]
so that from  \eqref{E:6g} we obtain
\begin{align*}
(\gL^2_{\mu +4} &\Phi (z,X)) \cdot (\gL^0_{\nu} \Psi (z,X))\\
&=\frac 1{(\nu-1)!} \biggl( \frac {\phi'' (z) \psi (z)} {2
  (\mu+1)!} + \frac {\phi' (z) \psi (z)} {\mu!} X + \frac {\phi (z) \psi
  (z)} {2 (\mu-1)!} X^2 \biggr) ,
\end{align*}
which is a quasimodular polynomial belonging to $QP^2_{\mu +\nu +4} (\G)$.
Thus, if we set
\[ G (z,X) = 2 (\nu-1)! (\mu+1)! (\gL^2_{\mu +4} \Phi (z,X)) ,\]
then we have
\[ G (z,X) = g_0 (z) + g_1 (z) X + g_2 (z) X^2 \in QP^2_{\mu +\nu +4}
(\G) ,\]%
where
\[ g_0 = \phi''  \psi = \fS_0 G \in QM^2_{\mu +\nu +4} (\G), \quad g_1 =
2(\mu+1) \phi' \psi , \quad g_2 = \mu (\mu+1)\phi  \psi .\]%
Using this, \eqref{E:5ow} and \eqref{E:8r}, we see that
\[ (\Xi^2_{\mu +\nu +4} G) (z,X) = \sum^2_{r=0}
g^\Xi_r (z) X^r ,\]
where
\begin{align*}
g^\Xi_0 &= 2 (\mu +\nu -1)! g_2
=  2\mu (\mu+1) (\mu +\nu -1)! [\phi, \psi]^{(\mu, \nu)}_0,\\
g^\Xi_1 &= (\mu +\nu +1) \Bigl( (\mu +\nu)! g_1 - 2 (\mu +\nu -1)!
g'_2 \Bigr)\\
&= 2 (\mu +\nu +1) (\mu +\nu -1)! \Bigl( \nu \phi' \psi - \mu
\phi  \psi' \Bigr)\\
&= - 2 (\mu +\nu +1) (\mu +\nu -1)! [\phi, \psi]^{(\mu, \nu)}_1 ,\\
g^\Xi_2 &= (\mu + \nu +3) \Bigl( (\mu + \nu +2)! g_0 - (\mu + \nu +1)!
g'_1 + (\mu + \nu)! g''_2 \Bigr)\\
&= (\mu + \nu +3) (\mu +\nu)! \Bigl( (\mu +
\nu +2) (\mu + \nu +1) \phi''  \psi\\
&\hspace{1.8in} - 2 (\mu +1) (\mu + \nu +1) (\phi''  \psi  +
\phi'  \psi')\\
&\hspace{2.2in} + \mu (\mu + 1) (\phi''  \psi  + 2\phi'
  \psi'  +\phi  \psi'') \Bigr)\\
&= (\mu + \nu +3) (\mu +\nu)! \Bigl( \nu (\nu +1) \phi'' \psi + 2 (\mu +1) (\nu +1) \phi' \psi' + \mu (\mu
+1) \phi \psi'' \Bigr)\\
&= 2 (\mu + \nu +3) (\mu +\nu)! [\phi, \psi]^{(\mu, \nu)}_2 .
\end{align*}
From this and \eqref{E:jf}, we have
\begin{align} \label{E:dd}
L (\phi'' \psi, s) &= \sum^2_{j =0} \frac {(2\pi
  i)^j} {j! (\mu +\nu +3 -j)! h^j} L (g^\Xi_{2 -j}, s -j)\\
&=  \frac {2} {(\mu +\nu +2) (\mu +\nu +1)} L ([\phi,
\psi]^{(\mu, \nu)}_2, s) \notag\\
&\hspace{.3in} - \frac {2 (2\pi i) (\mu+1)} {(\mu +\nu +2) (\mu +\nu) h}
L ( [\phi, \psi]^{(\mu, \nu)}_1, s -1)) \notag\\
&\hspace{.5in} + \frac {(2\pi
  i)^2 \mu (\mu+1)} {(\mu +\nu +1) (\mu +\nu) h^2} L ([\phi,
\psi]^{(\mu, \nu)}_0, s -2) . \notag
\end{align}

The case where $m=0$ and $n=2$ can be obtained from \eqref{E:dd} and the
fact that Rankin-Cohen brackets $[\phi, \psi]^{(\mu, \nu)}_w$ are
$(-1)^w$-symmetric.  Thus we see that
\begin{align*}
L (\phi \psi'', s) &=  \frac {2} {(\mu +\nu +2) (\mu +\nu +1)} L ([\psi,
\phi]^{(\mu, \nu)}_2, s)\\
&\hspace{.3in} - \frac {2 (2\pi i) (\nu+1)} {(\mu +\nu +2) (\mu +\nu) h}
L ( [\psi, \phi]^{(\mu, \nu)}_1, s -1)\\
&\hspace{.5in} + \frac {(2\pi
  i)^2 \nu (\nu+1)} {(\mu +\nu +1) (\mu +\nu) h^2} L ([\psi,
\phi]^{(\mu, \nu)}_0, s -2)\\
&=  \frac {2} {(\mu +\nu +2) (\mu +\nu +1)} L ([\phi,
\psi]^{(\mu, \nu)}_2, s)\\
&\hspace{.3in} + \frac {2 (2\pi i) (\nu+1)} {(\mu +\nu +2) (\mu +\nu) h}
L ( [\phi, \psi]^{(\mu, \nu)}_1, s -1)\\
&\hspace{.5in} + \frac {(2\pi
  i)^2 \nu (\nu+1)} {(\mu +\nu +1) (\mu +\nu) h^2} L ([\phi,
\psi]^{(\mu, \nu)}_0, s -2) .
\end{align*}

\medskip

\section{\bf{Concluding remarks}}

Given two modular forms $\phi \in M_{\mu} (\G)$ and $\psi \in M_{\nu}
(\G)$, from \eqref{E:bp} we see that
\[ L ([\phi , \psi]^{(k, \ell)}_w, s) = \sum^w_{r=0} (-1)^r \bi {k +w -1}
{w-r} \bi {\ell +w -1} {r} L (\phi^{(r)} \psi^{(w-r)}, s) \]
for each nonnegative integer $w$.  Using this and \eqref{E:ju}, we have
\begin{align*}
L ([\phi , \psi]^{(\mu, \nu)}_w, s) &= \sum^{w}_{\ell =0}
\sum^w_{r=0}
(-1)^r \bi {\mu +w -1} {w-r} \bi {\nu +w -1} {r}\\
&\hspace{1.2in} \times a^{r,w-r}_{\mu,\nu} (\ell) L ([\phi,
\psi]^{(\mu, \nu)}_{w -\ell}, s -\ell) ,
\end{align*}
where the constants $a^{r,w-r}_{\mu,\nu} (\ell) \in \bC$ are as
in \eqref{E:8q}.  Hence we obtain the identities
\[ \sum^w_{r=0} (-1)^r \bi {\mu +w -1} {w-r} \bi {\nu +w -1} {r}
a^{r,w-r}_{\mu,\nu} (0) =1, \]
\[ \sum^w_{r=0} (-1)^r \bi {\mu +w -1} {w-r} \bi {\nu +w -1} {r}
a^{r,w-r}_{\mu,\nu} (\ell) =0  \]
for $1 \leq \ell \leq w$.

\medskip

\providecommand{\bysame}{\leavevmode\hbox to3em{\hrulefill}\thinspace}
\providecommand{\MR}{\relax\ifhmode\unskip\space\fi MR }
\providecommand{\MRhref}[2]{%
  \href{http://www.ams.org/mathscinet-getitem?mr=#1}{#2}
}
\providecommand{\href}[2]{#2}


\begin{thebibliography}{1}

\bibitem{BGHZ06}J.~Bruinier, G.~van~der~Geer, G.~Harder and  D.~Zagier, \emph{
The 1-2-3 of modular forms,}
 Lectures from the Summer School on
Modular Forms and their Applications held in Nordfjordeid, June
2004. Edited by Kristian Ranestad. Universitext. Springer-Verlag,
Berlin, 2008. x+266 pp.



\bibitem{CL08}
Y.~Choie and M.~H.~Lee, \emph{Quasimodular forms, {J}acobi-like
forms, and
  pseudodifferential operators}, preprint.

\bibitem{CM97}
P.~B.~Cohen, Y.~Manin, and D.~Zagier, \emph{Automorphic
pseudodifferential
  operators}, Algebraic Aspects of Nonlinear Systems, Birkh\"auser, Boston,
  1997, pp.~17--47.

\bibitem{EO1}A.~Eskin and A.~Okounkov,  Asymptotics of numbers of branched
coverings of a torus and volumes of moduli spaces of holomorphic
differentials. Invent. Math. 145 (2001), no. 1, 59--103


 \bibitem{EO2}
 A.~Eskin and A.~Okounkov, Andrei Pillowcases and quasimodular forms.
 Algebraic geometry and number theory, 1--25, Progr. Math.,
 253, Birkhauser Boston, Boston, MA, 2006


\bibitem{KZ95}
M.~Kaneko and D.~Zagier, \emph{A generalized {J}acobi theta function and
  quasimodular forms}, Progress in Math., vol. 129, Birkh{\"{a}}user, Boston,
  1995, pp.~165--172.


\bibitem{LR} S.~Lelievre and  E.~Royer,   Orbitwise countings in $
H(2)$ and quasimodular forms. Int. Math. Res. Not. 2006, Art. ID
42151, 30 pp


\bibitem{L09b}
M.~H. Lee, \emph{Quasimodular forms and {P}oincar\'e series}, Acta Arith.
  \textbf{137} (2009), 155--169.

\bibitem{MR05}
F.~Martin and E.~Royer, \emph{Formes modulaires et transcendance}, Formes
  modulaires et p\'eriodes (S.~Fischler, E.~Gaudron, and S.~Kh\'emira, eds.),
  Soc. Math. de France, 2005, pp.~1--117.


\bibitem{Z94}D.~Zagier, \emph{Modular forms and differential operators},
K. G. Ramanathan memorial issue. Proc. Indian Acad. Sci. Math.
Sci. 104 (1994), no. 1, 57--75.



\end{thebibliography}
\end{document}